\documentclass[12pt]{amsart}
\usepackage{amsmath,amscd,amsthm,amsfonts,amssymb,amsxtra}
\usepackage[all]{xy}
\SelectTips{cm}{}
\newtheorem{thm}{Theorem}[section]  
\newtheorem*{un-no-thm}{Theorem}
\newtheorem{cor}[thm]{Corollary}     
  
\theoremstyle{definition}
\newtheorem{defn}[thm]{Definition}   

\theoremstyle{definition}

\theoremstyle{definition}

\theoremstyle{remark}
\newtheorem{rem}[thm]{Remark}

\newtheorem{rems}[thm]{Remarks}

\newtheorem{exs}[thm]{Examples}

\newcommand{\med}{\medskip}

\newcommand{\bz}{\mathbb{Z}}
\newcommand{\br}{\mathbb{R}}

\newcommand{\ce}{\mathcal{E}}

\newcommand{\cm}{\mathcal{M}}

\newcommand{\sr}{\mathcal{R}}

\newcommand{\ct}{\mathcal{T}}

\newcommand{\hk}{\hookrightarrow}
\newcommand{\bg}{\bigskip}
 
\newcommand{\bfl}{\begin{flushleft}}
\newcommand{\efl}{\end{flushleft}}

\newcommand{\hocolim}{\operatorname{hocolim}}

\newcommand{\xr}{\xrightarrow}

\begin{document}
\title{Umkehr Maps}
\date{\today}
\thanks{Both authors are partially supported by the National Science Foundation.}
\author{Ralph L.\ Cohen}
\address{Stanford University, Stanford CA 94305}
\email{ralph@math.stanford.edu}
\author{John R.\ Klein}
\address{Wayne State University, Detroit, MI 48202}
\email{klein@math.wayne.edu}

\begin{abstract}  In this note we study umkehr maps in generalized (co)homology theories arising from the Pontrjagin-Thom construction, from  integrating along fibers, pushforward homomorphisms, and other similar constructions.  We consider the basic properties of these constructions and develop axioms which any umkehr homomorphism must satisfy.  We use a version of Brown representability to   show
that these axioms completely characterize these homomorphisms, and a resulting uniqueness theorem follows.  Finally, motivated by constructions in string topology, we extend
this axiomatic treatment of umkehr homomorphisms to a fiberwise setting. 
\end{abstract}
\maketitle
\setlength{\parindent}{15pt}
\setlength{\parskip}{1pt plus 0pt minus 1pt}

\def\Top{\bold T\bold o \bold p}
\def\Sp{\text{\bf Sp}}
\def\vo{\varOmega}
\def\vs{\varSigma}
\def\smsh{\wedge}
\def\flush{\flushpar}
\def\id{\text{id}}
\def\dbslash{/\!\! /}
\def\codim{\text{\rm codim\,}}
\def\:{\colon}
\def\holim{\text{holim\,}}
\def\hocolim{\text{hocolim\,}}
\def\Bbb{\mathbb}
\def\bold{\mathbf}
\def\Aut{\text{\rm Aut}}
\def\cal{\mathcal}
\def\sec{\text{\rm sec}}

\section{Introduction}
The classical umkehr homomorphism of Hopf \cite{Hopf},
assigns to a map $f\: M \to N$
of closed manifolds of the same dimension a ``wrong way'' homomorphism
$f_!\: H_*(N)\to H_*(M)$ on singular homology. Hopf showed that 
this map is compatible with intersection pairings.
Freudenthal \cite{Freud} showed that $f_!$ corresponds
to the homomorphism $f^*\:H^*(N) \to H^*(M)$ induced by $f$ on
cohomology by means of the Poincar\'e duality isomorphisms
for $M$ and $N$. This identification allows one to give
a definition of the umkehr homomorphism for a map between
closed manifolds of any dimension.

Variants of the umkehr homomorphism, such as those defined 
by the Pontrjagin-Thom construction, intersections of  chains, 
integration along fibers, and the Becker-Gottlieb transfer, have 
played central roles in the development of differential and 
algebraic topology.  Similarly, the ``push-forward" constructions 
in cohomology, 
 Chow groups, and $K$-theory, have been important techniques in 
algebraic geometry and index theory.  Topological generalizations 
of umkehr mappings have played important roles in recent developments in topology, such as Madsen and Weiss's proof of the Mumford conjecture 
and its generalizations \cite{madsen-weiss}, 
\cite{GMTW}, \cite{galatius},  and the development of  
string topology \cite{chassullivan}, \cite{cohenjones}.

 Considering these various different, but related constructions,  it is natural to ask how they are related?  Similarly, one might ask: 
what properties characterize or classify umkehr homomorphisms?

The goal of this note is to describe naturally occurring axioms which completely classify umkehr homomorphisms. These axioms come as a result of considering the basic   properties of the  umkehr homomorphisms mentioned above.  We will show that a Brown-type representability theorem 
classifies these umkehr maps.     
In more recent applications, such as those in string topology, umkehr homomorphisms were needed  in the setting of pullback squares of Serre fibrations,
$$
\begin{CD}
E_1 @>\tilde f >>  E_2 \\
@VVV   @VVV \\
P @>>f > N 
\end{CD}
$$
where $f : P \to N$ is a smooth map of manifolds. That is, one wanted an umkehr homomorphism, $\tilde f_! : H_*(E_2) \to H_*(E_1)$ 
(with a dimension shift of $\dim N - \dim P$).   
This leads us to consider axioms for the existence 
and uniqueness of umkehr homomorphism in this fiberwise 
setting, using a fiberwise version of Brown representability,
 which we prove in the appendix.

\section{Preliminaries}
We will work in the category ${\cal T}$ of compactly generated
weak Hausdorff spaces. 
Products are to be re-topologized using the compactly
generated topology. Mapping spaces are to be given
the compactly generated, compact open topology.
A {\it weak equivalence} of
spaces denotes a (chain of) weak homotopy equivalence(s).

We will assume that the reader is familiar with
the standard machinery of algebraic topology including 
homotopy limits and colimits (the standard reference
for the latter is \cite{Bousfield-Kan}).

A spectrum $E$ is a sequence of
based spaces $E_n$, $n \ge 0$ together with (structure) maps
$\Sigma E_n \to E_{n+1}$, where $\Sigma$ denotes reduced suspension. 
A spectrum has homotopy groups $\pi_n(E)$ for $n \in {\Bbb 
Z}$ defined by the colimit of the system $\{\pi_{n+j}(E_j)\}_{j\ge 0}$.
A morphism of spectra $E\to E'$ consists of maps 
$E_n \to E'_n$ which are compatible with the structure maps.
A morphism is a weak equivalence if it induces an isomorphism in
homotopy groups in each degree. The category of spectra is
denoted ${\cal S}$.

We say that a spectrum $E$ is an {\it $\Omega$-spectrum}
if the adjoint maps $E_n \to \Omega E_{n+1}$ are weak equivalences.
For any spectrum $E$, there a weak equivalence 
$E\simeq  E^{\text{f}}$ with $E^{\text{f}}$ an  $\Omega$-spectrum. This
weak equivalence is natural.

For an unbased space $K$ we let $\text{\rm map}(K,E)$ denote
the mapping spectrum whose $j$-th space is given 
by  the space of (unbased) maps $K \to E_j$. The basepoint of this mapping space is taken to be the constant map at the basepoint of $E_j$.   The structure
maps in this case are induced by suspending and taking
the adjunction. For this to have the ``correct'' homotopy
type, it should be assumed that $E$ is an $\Omega$-spectrum and
that $K$ has the homotopy type of a CW complex.

Although it will not emphasized in the paper, the above discussion fits
naturally within the context of a Quillen model category structure
on the category of spectra (see for example, \cite{Schwede}).

\section{What should an umkehr map do?}
Umkehr homomorphisms are known to occur in all 
cohomology theories.  Now, every cohomology theory
is representable, so one can view the umkehr 
homomorphism as arising from a certain map of spectra. 

Minimally, an umkehr map should assign to
an embedding $f\: P \subset  N$ of closed manifolds a wrong way stable map
$$
f^! \: N^+ \to P^\nu
$$
where $\nu$ is the normal bundle
of $f$. One definition of $f^!$ is given by taking the Pontryagin-Thom
construction of the embedding $f$.

For a variety of reasons, it is also desirable to twist the above
by an arbitrary vector bundle $\xi$ over $N$. It this case, the umkehr map
should give a map of Thom spectra
$$
f^!_\xi \: N^\xi \to P^{\nu+ \xi}\, .
$$
Classically, such an $f^!_\xi$ is  produced by taking the Pontryagin-Thom
construction of the composite
$$
\begin{CD}
P \subset N @> \text{zero section} >> D(\xi)
\end{CD}
$$
where $D(\xi)$ is the unit disk bundle of $\xi$.
This directly motivates one to consider umkehr maps as being
defined not only for closed manifolds, but more generally 
for maps of compact manifolds {\it having a boundary.}
The twisting by bundles then becomes a special case, as
$D(\xi)$ is a manifold with boundary.

For example,
if $f$ fails to be an embedding, we can always approximate
the composite 
$$
f_j\: P \subset N = N \times 0 \subset N \times D^j 
$$
by an embedding when $j$ is sufficiently large, and therefore,
assuming umkehr maps have been defined for manifolds with boundary,
we obtain an umkehr map for $f_j$ which we simply declare to 
be the umkehr map for $f$.

The above suggests the following.
Let ${\cal M}$ be the category whose objects are compact
manifolds $P$ (possibly with boundary) in which a morphism
$P \to Q$ is a continuous map (not necessarily preserving the boundary).
Let ${\cal S}$ be the category of spectra. We will consider 
{\it contravariant} functors
$$
u\: {\cal M}  \to {\cal S}
$$
satisfying certain axioms.

The first two axioms  will be

\begin{itemize}
\item {\it Vacuum Axiom.} If $\emptyset$ is
the empty manifold, then $u(\emptyset)$ is
contractible.

\item {\it Homotopy Invariance Axiom.} If $f\: P \to Q$
is a weak (homotopy) equivalence, then so is
 $u(f)$.
\end{itemize}
The vacuum axiom is motivated by the
fact that Pontryagin-Thom collapse of an
empty submanifold yields a constant map.

The homotopy invariance axiom is motivated by the following. Let
$P\subset N$ be a homotopy equivalence, where $P$ is closed.
 Then 
the Pontryagin-Thom collapse $N/\partial N \to P^\nu$ is a 
stable homotopy equivalence.

The last axiom  umkehr functors are required to satisfy is {\it locality}.
In its most geometric form, locality will mean that
a decomposition of manifolds yields a corresponding decomposition
of their Pontryagin-Thom collapses.
Suppose for example that
$P \subset S^n$ is a closed submanifold with normal
bundle $\nu$ such that $P$ is transverse
to the equator $S^{n-1} \subset S^n$. Let $D^n_\pm$
denote the upper and lower hemispheres. 
Setting $P_\pm = P \cap D^n_\pm$ and $Q = P \cap S^{n-1}$,
we obtain a decomposition $P = P_- \cup_Q P_+$.
The Pontryagin-Thom collapse of each inclusion $(P_\pm,Q) \subset 
(D^n_\pm,S^{n-1})$
gives maps
$$
k_\pm\:(D^n_\pm,S^{n-1}) \to (P_\pm^\nu,Q^\nu)
$$
which may be glued to a yield a map 
$$
\begin{CD}
S^n = D^n_- \cup_{S^{n-1}} D^n_+ @> k_- \cup k_+ >>
 P_-^\nu \cup_{Q^\nu} P_+^\nu = 
P^\nu 
\end{CD}
$$
which is just the Pontryagin-Thom construction of $P \subset S^n$.

In general, it seems that
the cleanest way to formulate the locality axiom is in terms
of a (left homotopy) Kan extension of $u$
to the category ${\cal T}$ of topological spaces.
The resulting functor will also be homotopy invariant. The 
Kan extension $u^\#\: {\cal T}\to {\cal S}$ is the contravariant functor given by 
$$
u^\#(Y) = \underset{P \overset\sim\to Y}{\text{hocolim }} u(P)\, ,
$$
where the homotopy colimit is indexed over the category 
of compact manifolds $P$ equipped with a weak equivalence to $Y$.
Notice that $u^\#$ restricted to ${\cal M}$
coincides with $u$ up to natural equivalence.

\begin{itemize}
 \item {\it Locality Axiom.} The functor $u^\#$ is excisive, i.e.,
it preserves homotopy cocartesian squares.
\end{itemize}

The above axioms  imply, by a version of Brown's representability theorem
(cf.\ appendix), that
the composite $u^\#$ is representable: 
there is an $\Omega$-spectrum
$E$, unique up to weak equivalence, and a natural weak equivalence
$$
u^\#(X) \,\, \simeq \,\, \text{map}(X,E)\,
$$
where $\text{map}(X,E)$ denotes the spectrum of (unbased) maps from
$X$ to $E$, i.e., the spectrum whose
$j$-th space is the space of unbased maps $X \to E_j$. 
(In the above, we are implicitly using the fact that every
compact manifold has the homotopy type of a finite complex. This
will imply that $u^\#$ is determined up to equivalence by its
restriction to the category of finite complexes over $X$.)

Notice that we can recover $E$ by taking 
of $u(*)$, where $*$ is the one-point manifold.
Summarizing,

\begin{thm} An umkehr functor $u\: {\cal M} \to {\cal S}$ is
characterized up to natural weak equivalence
by its value $E:= u(*)$ at the one point manifold.

Conversely,
an $\Omega$-spectrum $E$ gives rise to 
an umkehr functor $u$ by the rule 
$$
u(P) \,\, := \text{\rm map}(P,E)\, .
$$
\end{thm}

\subsection*{Examples}

\subsubsection*{Example 1: The Pontryagin-Thom construction}  
The traditional Pontryagin-Thom construction comes 
from the umkehr functor corresponding to the spectrum  $E = u(*) = S^0$, 
the sphere spectrum.  That is,
$$
u(P) = \text{map}(P, S^0),
$$
the Spanier-Whitehead dual of $P$.  
The fact that this functor yields the Pontrjagin-Thom construction 
comes from Atiyah duality, which gives a natural equivalence of spectra,
$$
\text{\rm map}(P, S^0) \simeq P^{-\tau_P}
$$
where  $P^{-\tau_P}$  is the Thom spectrum of the stable normal bundle, that is, the virtual bundle $-\tau_P$, where $\tau_P$ is the tangent bundle of $P$.

Given an embedding $f\: P \to N$ with normal bundle $\nu(f)$,  
the map of Spanier-Whitehead duals, 
$f_! \: \text{\rm map}(N, S^0) \to \text{\rm map}(P, S^0)$ 
is equivalent to the Pontryagin-Thom map
$$
N^{-\tau_N} \to P^{-\tau_N \oplus \nu (f)} = P^{-\tau_P}.
$$

\subsubsection*{Example 2: Integration along the fibers}
Consider a smooth submersion of closed oriented manifolds,
$$
P^{n+k} \overset p\to M^n
$$
where the superscript denotes dimension.  Then integration along fibers defines a homomorphism in de\,Rham cohomology,
$$
p^{\int} : H^q_{\text{dR}}(P) \to H^{q-k}_{\text{dR}}(M).
$$
This can be seen in terms of the umkehr functor defined 
by setting $u(*) = h{\Bbb R}$, the Eilenberg-MacLane spectrum 
for the real numbers.  In other words, $u(N) = \text{map}(N, h{\Bbb R})$.  
The homomorphism induced by the bundle gives a homomorphism,
$$
\text{\rm map}(M, h{\Bbb R}) \overset{p_!}\to \text{\rm map}(P,  h{\Bbb R})
$$
which can be written as $p_! :\text{\rm map}(M, S^0)\wedge h{\Bbb R} \to 
\text{\rm map}(P, S^0)\wedge h{\Bbb R}$.  Using Atiyah duality, 
as in the previous example,  this is equivalent to a map 
which, by abuse of notation, we also call $p_!$,
$$
p_!\: M^{-\tau_M}\wedge h{\Bbb R} \to P^{-\tau_P}\wedge h{\Bbb R}.
$$
When one applies homotopy groups to this map, 
and the Thom isomorphism, one obtains a homomorphism,
$$
p_! \: H_{q-k}(M; {\Bbb R}) \to H_{q}(P; {\Bbb R})
$$which is linearly dual to the integration map $p^{\int}$.  That is, 
$$
\int_\alpha p^{\int} (\omega)  = \int_{p_!\alpha}\omega.
$$

\subsubsection*{Example 3: Oriented bordism}
For a space $X$, let $\text{\rm MSO}_p(X)$ denote bordism
classes of maps $P \to X$ where $P$ is a closed smooth
oriented $p$-manifold.
If  
$$
f\:Q \to  N 
$$
is a map of closed smooth oriented 
manifolds, then we obtain an umkehr homomorphism
$$
f_*^!\:\text{\rm MSO}_p(N) \to \text{\rm MSO}_{p+q-n}(Q)
$$
as follows.  Let $\gamma \in \text{\rm MSO}_p(N)$.   Choose a  representative $g\: P \to N$ 
of $\gamma$ in such
a way that $f$ and $g$ are mutually transverse. Then the fiber product
$P \times_N Q$ is an oriented manifold of dimension $p+q-n$, and the bordism 
class of evident
map $P \times_N Q \to Q$ defines the umkehr homomorphism.  Of course the spectrum representing the associated umkehr functor is the Thom spectrum, $\bf MSO$. \rm

\section{A generalization}
A generalization of the umkehr map arises
naturally within the framework of string topology.

The context is this:
one has an embedding $P \subset N$ of closed manifolds
with normal bundle $\nu$, and also a  (not necessarily smooth)  fiber bundle
$p\:E\to N$. Let $q\:E_{|P}\to P$ be the restriction of 
$p$ to $P$. Then we have a cartesian square
$$
\xymatrix{
E_{|P} \ar[r] \ar[d] & E \ar[d]\\
P \ar[r] & N\, .
}
$$
The spaces $E$ and $E_{|P}$  may not be smooth, and may even be infinite dimensional, (for example in string topology the total space $E$ is typically built from path or loop spaces in the manifold $N$). However   one 
still observes that the codimension of $E_{|P}$ in $E$ is finite, and
that one can find a regular neighborhood which is homeomorphic to 
the pullback of $\nu$ along $q$. Collapsing a complement
of this tubular neighborhood to a point, we obtain
a based map
$$
E^+ \to (E_{|P})^{q^*\nu}
$$
where the target is the Thom space of $q^*\nu$. Given this
construction, which seems depend on certain choices,
it is not entirely clear that it carries with it any
uniqueness properties. We will show in fact that it does. 
\medskip

It will be convenient for us to categorify the above.
The idea will be that the above umkehr map
can be thought of as arising from a suitable functor on the  
category of manifolds {\it over} $N$.  The representing objects in this setting will be fiberwise spectra.  For the sake of completeness, we begin with a digression describing those aspects of fiberwise spectra that we will need for our purposes.  The reader is referred to \cite{May-Sigurdsson} for a more complete discussion.

\subsection{Fibered Spectra}
One may regard a spectrum as a generalization of
an abelian group, where the latter appear
as the Eilenberg-Mac\, Lane spectra. 
Analogously
a fibered spectrum on a space $X$ 
can be thought of as a bundle of local coefficients on $X$ 
in which the fibers, which were formerly abelian groups, 
are now replaced by spectra. 
\medskip

For a space $X$, let  ${\cal R}_X$
be the category of {\it retractive spaces} over $X$. An object 
is a space $Y$ equipped with maps $s_Y\:X\to Y$ and $r_Y\: Y\to X$
such that $r_Y\circ s_Y$ is the identity map (the structure maps
$r_Y, s_Y$ are usually suppressed from the notation). 

A morphism $f\:Y\to Z$ is a
map of underlying spaces which commutes with their structure maps:
$r_Z\circ f = r_Y$ and $f\circ s_Y = s_Z$.
A morphism is a {\it weak equivalence} 
if it is a weak homotopy equivalence of underlying spaces.

One has a forgetful functor $u\:{\cal R}_X \to {\cal T}_X$.
There is also a functor $v$ in the other direction given by
$Y \mapsto Y^+$, where $Y^+$ is the retractive space $Y \amalg X$.
One readily verifies that $u$ is a right adjoint to $v$.

Given objects $Y,Z\in {\cal R}_X$, the hom-set 
$\hom_{{\cal R}_X}(Y,Z)$ may be topologized as a subspace
of the function space of all continuous maps $Y \to Z$ of underlying spaces,
where the function space is equipped with the compactly generated compact
open topology. This gives 
${\cal R}_X$ the structure of a of a topological category.
\medskip

\begin{defn}[Fiberwise suspension]
Given an object $Y \in {\cal T}_X$, its {\it unreduced fiberwise
suspension} is defined to be the double mapping cylinder
$$
S_X Y \,\, := \,\, X \times 0 \cup Y\times [0,1] \cup X \times 1\, .
$$
It comes with an evident map $S_X Y \to X$, so it is an object
of ${\cal T}_X$.

Given an object $Y \in {\cal R}_X$, its 
{\it reduced fiberwise suspension} is given by 
$$
\Sigma_X Y \,\, = \,\,  S_X Y \cup_{S_X X} X
$$
Note that $\Sigma_X$ defines an endo-functor of ${\cal R}_X$.

If $Y, Z$ are objects of ${\cal R}_X$, its fiberwise smash product 
$Y\smsh_X Z$ is the object given
the pushout of the diagram
$$
X \leftarrow Y\cup_X  Z \to  Y \times_X Z  \, .
$$
\end{defn}

\begin{defn}[Fibered spectra]\label{fibered}
A {\it fibered spectrum} 
$ {\cal E}$ over $X$
consists of  objects ${\cal E}_j \in {\cal R}_X$ for $j \in {\Bbb N}$ together
with (structure) maps
$$
\Sigma_X {\cal E}_j \to {\cal E}_{j+1} \,  ,
$$
for each $j \ge 0$.
A {\it morphism} ${\cal E}\to {\cal E}'$ is given by   
maps ${\cal E}_j \to {\cal E}'_j$ which are compatible with the
structure maps. 
\end{defn}

We say that ${\cal E}$ is {\it fibrant} if the adjoints to the structure
maps are weak homotopy equivalences of underlying spaces.
Any fibered spectrum ${\cal E}$ can be converted into a fibrant
one ${\cal E}^{\text{f}}$ in which 
$$
{\cal E}^{\text{f}}_j\,\, := \,\, \underset{n}{\text{hocolim\, }} 
\Omega^n_X {\cal E}_{j+n} \, ,
$$
where the homotopy colimit is taken in ${\cal R}_X$, and 
$\Omega_X^j$ is the adjoint to $n$-fold reduced fiberwise
suspension. The above is called {\it fibrant 
replacement.}

A morphism ${\cal E}\to {\cal E}'$ is a 
{\it weak equivalence} if the associated morphism of 
fibrant replacements  ${\cal E}^\text{f}\to ({\cal E}')^\text{f}$ 
is a {\it levelwise} weak equivalence: 
for each $j$, the map ${\cal E}_j^\text{f}\to ({\cal E}')_j^\text{f}$ 
is required to be weak equivalence of ${\cal R}_X$. 

\med
\subsubsection*{Examples}

\med
\noindent
{(1). Fiberwise suspension spectra} 

Let $Y \in {\cal R}_X$ be an object. Let $\Sigma^\infty_X Y$ be
the fibered spectrum over $X$ given by the collection $\Sigma^j_X Y$ of
iterated fiberwise suspensions of $Y$.

\med

\noindent{(2). Trivial fibered spectra} 

Let $C$ be a spectrum. The collection
of spaces $C_j \times X$ as $j$-varies forms a fibered spectrum over $X$. 
The maps  $\Sigma_X (C_j \times X) \to C_{j+1} \times X$ 
use the identification $\Sigma_X (C_j \times X) = (\Sigma C_j) \times X$
together with the structure maps of $C$. 

\med

\noindent{(3). Fibered Eilenberg-Mac\,Lane  spectra}

Let ${\cal F}$ be a bundle
of abelian groups on $X$
Let ${\cal F}_x$ denote the fiber at $x$.
Then we have a fibered spectrum $h{\cal F}$ on $X$, in which 
$h{\cal F}_j$ can be described as follows: the fiber at $x \in X$ is given by 
$K({\cal F}_x,j)$, the Eilenberg-Mac\,Lane space based on the
abelian group ${\cal F}_x$.

\med

\noindent{(4). Fiberwise smash product with a spectrum}

Let $C$ be a spectrum and a let $E\to X$ be 
a fibration. Then we obtain a fibered spectrum $E\otimes C$ in
whose $j$-th total space is given by the pushout of the diagram
$$
\begin{CD} 
X @<<< E \times *@>\subset >>  E \times C_j \, .
\end{CD}
$$
If $E_x$ is the fiber to $E\to X$ at $x\in X$, then the fiber
of $(E\otimes C)_j \to X$ is given by $(E_x)_+ \smsh C$.

\med
\noindent{(5). Twisted suspension} Let ${\cal E}$ be a fibered spectrum over
$X$. If $\xi$ is a vector bundle over $X$ we can form a new
fibered spectrum  
$ 
{}^\xi \!{\cal E} 
$ 
called the {\it twist} of ${\cal E}$ by $\xi$. The $j$-th total space
of ${}^\xi {\cal E}$ takes the form of a fiberwise smash product
 $$S^\xi \smsh_X {\cal E}_j ,$$
 where $S^\xi$ is the object of ${\cal R}_X$ given by the fiberwise
one point compactification of $\xi$.
The notion of  twisting extends to the case when $\xi$ is a virtual bundle
(we omit the details, but see example (2) below of the Poincar\'e duality equivalence (Theorem \ref{poincare})).

\eject
 
\noindent \bf {Homology}   \rm

\med
A fibered spectrum ${\cal E}$ gives rise to a 
covariant spectrum-valued functor 
$$
H_\bullet({-};{\cal E})\: {\cal T}_X \to {\cal S}
$$
called {\it homology} with ${\cal E}$-coefficients. 

Consider the following construction: let $Y \in {\cal T}_X$ be an object and
call the structure map $f\: Y \to X$. Let $f^*{\cal E}$
be the fibered spectrum over $Y$ given by the collection of fiber products
$Y \times_X {\cal E}_j$. The set of quotient spaces
$$
(Y \times_X {\cal E}_j)/Y 
$$
yields a spectrum. However,
we must take the derived version of this construction
to insure homotopy invariance.

Here are the details. First of all, we need to replace ${\cal E}$ with
its fibrant replacement ${\cal E}^{\text f}$. Secondly, we must replace
the above quotient, by a homotopy quotient, i.e., the mapping cone. The result
of these changes will produce a spectrum with $j$-th space
$$
(Y \times_X {\cal E}^{\text f}_j) \cup_Y CY \, .
$$
This spectrum is $H_\bullet(Y;{\cal E})$.

\begin{exs} The homology spectrum of the trivial fibered spectrum (example (2)
above)
 is $C \smsh Y_+$.

For the fibered Eilenberg-Mac\,Lane spectrum $h{\cal F}$ (example (3) above)
The homology spectrum of $Y$ has homotopy groups isomorphic
to the homology $Y$ with coefficients in the bundle of coefficients
${\cal F}$ pulled back to $Y$. 
\end{exs}

\noindent \bf {Cohomology}  \rm

\med
Given a fibered spectrum ${\cal E}$,
we obtain a contravariant spectrum-valued functor
$$
H^\bullet({-};{\cal E}) \:{\cal T}_X \to {\cal S}
$$
called {\it cohomology} with ${\cal E}$-coefficients. 
Roughly, it is given at an object $Y$ by taking the spectrum of sections
of ${\cal E}$ along $Y\to X$.

More precisely, consider the spectrum whose $j$-th space
is the hom-space $\hom_{{\cal T}_X}(Y, {\cal E}_j)$
(or equivalently, the space of sections of ${\cal E}_j \to X$
along $Y$). The structure maps for ${\cal E}$ yield structure
maps on these hom-spaces, so we obtain a spectrum.

To get a homotopy invariant version of this construction, we
need to replace ${\cal E}$ by its fibrant replacement, and
$Y$ by a functorial cellular approximation (for example,
we can replace $Y$ by $|SY|$, the realization of the simplicial
total singular complex of $X$). The result of these
manipulations yields a spectrum $H^\bullet(Y;{\cal E})$
which is homotopy invariant in $Y$.

\bg
 
\noindent \bf  {Poincar\'e duality}  \rm 

\med
Let $N$ be a closed manifold of dimension $d$ 
with tangent bundle $\tau_N$. 
 Let $-\tau_N$ be the virtual bundle of dimension $-d$ representing the
stable normal bundle of $P$.

We now state the Poincar\'e duality theorem with coefficients
in a fibered spectrum. 

\begin{thm}[Poincar\'e Duality]\label{poincare} For any
fibered spectrum ${\cal E}$ over $N$, there is a weak equivalence of spectra
$$
H_\bullet(N; {}^{-\tau_N}\!{\cal E}) \,\, \simeq \,\,  
H^\bullet(N;{\cal E})\, .
$$
The equivalence is natural in ${\cal E}$.
\end{thm}

Although usually stated differently,
this result appears in the literature (see
\cite[thms.\ A,D]{Klein_dualizing}, \cite[\S5,8]{Klein_dualizing_2}, 
\cite[th.\ 4.9]{PoHu}, \cite[prop.\ 2.4]{WW1}).

\begin{defn} A closed $n$-manifold $N$ is {\it ${\cal E}$-orientable}
if there is a weak equivalence of fibered spectra
$$
{}^{-\tau_N}\!{\cal E} \,\, \simeq \,\, {\cal E}[-n]
$$
where ${\cal E}[-n]$ is the $n$-fold fiberwise desuspension of
${\cal E}$.
\end{defn}

\begin{cor} Assume $N$ is ${\cal E}$-orientable. Then there is
a weak equivalence of spectra
$$
H_\bullet(N; {\cal E}[-n]) \,\, \simeq\,\, 
H^\bullet(N;{\cal E}) \, .
$$
\end{cor}

\subsubsection*{Examples}

\med
\noindent
{(1).  (\sl Atiyah and Spanier-Whitehead duality)  \rm  

\med
Let $\cal E$ be the trivial suspension
spectrum, $\Sigma^\infty_NN$.  In other words, the $j^{th}$-space is given by
$$
(\Sigma^\infty_NN)_j = S^j \times N \to N
$$
so the fiber over any point is the sphere $S^j$.  We can describe the twisted spectrum,
$
{}^{-\tau_N}\!{(\Sigma^\infty_NN)}$ in the following way.  Suppose we have an embedding in Euclidean space,  $N \hookrightarrow \br^L$, with normal bundle $\nu_L \to N$. Then for any $j \geq 0$, the $(j+L)^{th}$ space of the twisted spectrum is given by
$$
{}(^{-\tau_N}\!{\Sigma^\infty_NN})_{j+L} = S^{\nu_L}\wedge_{N}(\Sigma^\infty_NN)_j = S^{\nu_L}\wedge S^j.
$$
Then clearly the homology spectrum, $$H_\bullet(N; {}^{-\tau_N}\!{\Sigma^\infty_NN }) = N^{-\tau_N}
$$
the Thom spectrum of the virtual bundle $-\tau_N$.   On the other hand, the cohomology spectrum,
$H^\bullet(N; \, \Sigma^\infty_NN)$  has as its $j^{th}$-space the space of sections,
$$ \hom_{\ct_N}(N,  S^j \times N) =\text{\rm map}(N, S^j). $$   In other words, this cohomology spectrum is the  mapping  spectrum $\text{\rm map}(N, S^0)$, 
or the Spanier-Whitehead dual of $N_+$.  
Thus the Poincar\'e duality equivalence Theorem \ref{poincare} in this case gives the Atiyah duality,
$$
N^{-\tau_N} \simeq \text{map}(N , S^0).
$$

\med
 (2).  (\sl The free loop space and string topology)  \rm 
 
 \med
   Let $LN = \text{\rm map}(S^1, N)$ be the 
free loop space, and let $e : LN \to N$ be 
the fibration that evaluates a loop at the basepoint 
$0 \in \br/\bz = S^1$. The fiber at $x_0$ is the based loop space, 
$\Omega_{x_0}N$.   There is a section $\sigma : N \to LN$ of this fibration by
 considering a point $x \in N$ as the constant loop at $x$.   We consider the fiberwise suspension spectrum, $\ce = \Sigma^\infty_N LN$.  This fibered spectrum  has as its $j^{th}$ space  the $j$-fold fiberwise reduced suspension,
 $\Sigma^j_NLN$, which fibers over $N$, with fiber $\Sigma^j (\Omega N)$.  We consider  the Poincar\'e duality equivalence (Theorem \ref{poincare}) in the case of this fibered spectrum.  
     
     We consider the twisted spectrum ${}^{-\tau_N}\!{(\Sigma^\infty_NLN)}$.  This fibered   spectrum can be described in the following way.  Suppose, as above, $N \hk \br^L$ with normal bundle $\nu_L \to N$.  Then for any $j \geq 0$, the $(j+L)^{th}$ space of the twisted spectrum is given by
$$
{}(^{-\tau_N}\!{\Sigma^\infty_NLN})_{j+L} = S^{\nu_L}\wedge_{N}(\Sigma^j_NLN).   
$$
Then clearly the homology spectrum is given by, 
\begin{equation}\label{loophom}
H_\bullet(N; {}^{-\tau_N}\!{\Sigma^\infty_NLN }) = LN^{-\tau_N}
\end{equation}
the Thom spectrum of the virtual bundle $e^*(-\tau_N)$.   It was shown in \cite{cohenjones} that the spectrum $LN^{-\tau_N}$ is a ring spectrum, whose induced product in homology reflects the Chas-Sullivan loop product in \sl string topology \rm \cite{chassullivan} after one applies the Thom isomorphism.  This product can be seen by applying the Poincar\'e duality equivalence (Theorem \ref{poincare}) as follows.   

The cohomology spectrum,
$H^\bullet(N; \, \Sigma^\infty_NLN)$  has as its $j^{th}$-space the space of sections,
$\hom_{\ct_N}(N, \Sigma^j_N LN).$  We therefore write this spectrum as 
$\hom_{\ct_N}(N, \Sigma^\infty_NLN)$.   
The Poincar\'e duality equivalence in this setting gives an equivalence,
\begin{equation}\label{compare}
LN^{-\tau_N} \simeq \hom_{\ct_N}(N, \Sigma^\infty_NLN).
\end{equation}
  Now notice that the fiberwise spectrum $\Sigma^\infty_NLN$ is a fiberwise ring spectrum, since the fibration $\Omega N \to LN \to N$ is a fiberwise monoid.  ( More precisely it is  a fiberwise $A_\infty$-monoid.  See \cite{gruhersalvatore}.)  Thus the spectrum of sections, $\hom_{\ct_N}(N, \Sigma^\infty_NLN)$ is a ring spectrum.  This ring spectrum structure reflects
  the ring structure in $LN^{\tau_N}$, and thus reflects the string topology loop product.
}


\medskip
\subsection{Generalized umkehr functors}
 
Let $X$ be a topological space.
Let ${\cal M}_X$ be the category whose objects
are compact manifolds $P$ (possibly
with boundary) equipped with a map $P \to X$; the
map will not usually be specified in the notation. A morphism
is a map $f\:P \to Q$ which is compatible with maps to $N$ in the
obvious way (again, we do not require that
$f$ preserves boundaries). 
A morphism is a weak equivalence if and only if the
underlying map of spaces is a weak homotopy equivalence.

We will  consider
contravariant functors
$$
u\: {\cal M}_X \to {\cal S} \, .
$$

\begin{defn}
A  functor $u$ will be called  a {\it generalized umkehr functor}
if it satisfies three axioms. The first two axioms are:
\begin{itemize}
\item {\bf Axiom 1} (Vacuum). The value of
$u$ at the empty manifold $\emptyset$ is contractible. 
\item {\bf Axiom 2} (Homotopy Invariance). 
$u$ is a homotopy functor, i.e., if 
a morphism $f\: P \to Q$
is a weak (homotopy) equivalence, then so is
 $u(f)$.
\end{itemize}
\end{defn}

Let ${\cal T}_X$ be the category of spaces
over $X$. An object of this category consists
of a space $Y$ together with map $Y\to X$ (the latter
not usually specified). A morphism $Y\to Z$ is map
of underlying spaces that is compatible with maps to $X$.
As before, we can perform a left homotopy Kan extension
to $u$ along the  full inclusion ${\cal M}_X \subset {\cal T}_X$
to obtain a contravariant homotopy functor
$$
u^\#\: {\cal T}_X \to {\cal S}\, .
$$
The final axiom for generalized umkehr functors is
\begin{itemize}
\item {\bf Axiom 3} (Locality). 
The functor $u^\#$ preserves homotopy cocartesian squares.
\end{itemize}
(a square of ${\cal T}_X$ is homotopy cocartesian if
it is one when considered in ${\cal T}$ by means of the forgetful
functor.)

Again, we see that these axioms imply that  
$u^\#$ is representable. (The appropriate fiberwise version of Brown representability will be proved in the appendix.)  In this fiberwise setting,  representability means  there is a
fibered spectrum ${\cal E}$, unique up to
equivalence, and a natural weak equivalence
$$
u^\#(X) \,\,  \simeq \,\, H^\bullet(X;{\cal E})\, .
$$
In particular ${\cal E}$ and  $u$ determine one another. 
Summarizing,

\begin{thm} A fibered spectrum ${\cal E} \to N$ gives rise to
an umkehr functor by the rule
$$
u(P) \,\, := \,\, H^\bullet(P; {\cal E})\, .
$$

Conversely, a functor $u$ that satisfies axioms 1-3
determines a fibered spectrum ${\cal E} \to N$,
unique up to weak equivalence, whose associated
cohomology recovers $u$ up to natural equivalence.
\end{thm}

\subsection*{The generalized umkehr homomorphism}
Let ${\cal E} \to X$ be a fibered spectrum, and suppose
$$
f\: P \to Q
$$
is a morphism of ${\cal M}_X$ such that $P$ and $Q$ are closed manifolds.
We then have an induced map on cohomology spectra
$$
f^\bullet\:H^{\bullet}(Q;{\cal E}) \to H^{\bullet}(P;{\cal E})
$$
using the Poincar\'e duality equivalence, we can rewrite this
up to homotopy as a map
$$
f^!\: H_\bullet(Q;{}^{-\tau_Q}\!{\cal E})\to
H_\bullet(P;{}^{-\tau_P}\!{\cal E})
$$
Assume now that $P$ and $Q$ are ${\cal E}$-oriented. Then
taking homotopy groups of $f^!$, we get a homomorphism
$$
f^!_*\: H_*(Q;{\cal E})\to
H_{*+q-p}(P;{\cal E})\, .
$$
This is the generalized umkehr homomorphism.

\subsection*{Umkehr maps in string topology}

As seen in Example 2 of the Poincar\'e duality equivalence, the basic ring
structure structure arising in string topology can be seen via the equivalence (\ref{compare}) of $LM^{-\tau_M}$ and the ring spectrum, $\hom_{\ct_M}(M, \Sigma^\infty_MLM)$. Here $M$ is a closed manifold.    However in its original form \cite{chassullivan} and \cite{cohenjones}, the string topology product was created via an umkehr map.    We now see how this fits
into our framework.

  $L^\infty M $ be the space of maps from the figure eight
$S^1 \vee S^1$ to $M$. This space is the fiber product, $LM \times_M LM$.  That is, we have a   a pullback square
$$
\begin{CD}
L^\infty M @>\subset >> LM \times LM \\
@VeVV @VVe \times eV \\
M @>> \Delta > M \times M
\end{CD}
$$
where $\Delta$ is the 
a diagonal map, the vertical maps of the square are
the fibrations given by evaluation at the basepoint, and
the upper horizontal map arises from the quotient map
$S^1 \amalg S^1 \to S^1 \vee S^1$ by taking maps into $M$.

Let $h{\Bbb Z}$ be the Eilenberg-Mac\,Lane spectrum on the integers. Consider the product
$$
LM \times LM \xr{e \times e} M \times M
$$
as an object in $\sr_{M\times M}$.  We consider the  fiberwise smash product
spectrum $\ce = (LM \times LM) \otimes h\bz$  (See example (4)  after Definition (\ref{fibered}) above.)

In particular, $(LM \times LM) \otimes h\bz$  is  the fibered spectrum whose $j$-th space
is given by the pushout of the diagram
$$
\begin{CD}
M \times M @<<< LM\times LM \times * @> \subset >> 
L M \times LM \times (h{\Bbb Z})_j \, .
\end{CD}
$$

The umkehr homomorphism taken with respect to the fibered spectrum $(LM \times LM) \otimes h\bz$, applied to the diagonal map $\Delta : M \to M \times M$ viewed as a morphism in $\cm_{M\times M}$  is then computed by the induced map in cohomology spectra,
\begin{align}\label{stringcoho}
  H^{\bullet}(M\times M;{(LM \times LM) \otimes h\bz }) &\xr{\Delta^{\bullet}} H^{\bullet}(M;{(LM \times LM) \otimes h\bz})  \\
 \hom_{\ct_{M\times M}}(M \times M, (LM \times LM) \otimes h\bz )  &\xr{\Delta^*}   \hom_{\ct_{M\times M}}(M , (LM \times LM) \otimes h\bz )   \notag \\
 &= \hom_{\ct_M}(M, L_\infty M\otimes h \bz), \notag
 \end{align}
 and then apply the Poincar\'e duality equivalence (Theorem \ref{poincare}).  
But by an argument   completely analogous to that used to verify (\ref{loophom}) in Example (2) of the Poincar\'e duality equivalence, we see that
 
$$
H_\bullet(M\times M ; {}^{-\tau_{M\times M}}\!{((LM \times LM) \otimes h\bz })) = LM^{-\tau_M} \wedge LM^{-\tau_M}\wedge h\bz    
$$
If $M$ is oriented in singular homology, this last spectrum is equivalent to 
$\Sigma^{-2m}(LM \wedge LM)_+ \wedge h\bz$.    Similarly, from the above pull back square we see that
$$
H_\bullet(M  ; {}^{-\tau_{M\times M}}\!{((LM \times LM) \otimes h\bz })) = L_\infty M^{-\tau_M} \wedge h\bz
$$
where the last spectrum is equivalent to $\Sigma^{-m}(L_\infty M)_+ \wedge h\bz$  assuming $M$ is equipped with an orientation. Thus the umkehr map in this situation
gives a map
$$
\Sigma^{-2m}LM  \wedge LM \wedge h\bz  \to \Sigma^{-m}L_\infty M  \wedge h\bz,
$$ or by taking homotopy groups this 
takes the form
$$
H_*(L M \times LM ) \to H_{*-m}(L^\infty M) \, .
$$
The Chas-Sullivan loop product is given by the composite
\begin{align}
 H_{p}(L M)\otimes H_q(LM) &\to  H_{p+q}(L M \times LM )  \notag \\
&\to
 H_{p+q-m}(L^\infty M) \to H_{p+q-m}(LM), \notag
\end{align}
where the first homomorphism 
 is the external product, the second is the umkehr homomorphism
and the third is given by taking
the homology of the map of spaces $L^\infty M \to L M$  arising from
 pinch map $S^1 \to S^1 \vee S^1$.

\section{Appendix: Representability}

The purpose of this section is to outline a proof the representability theorem
for contravariant functors $f\:{\cal T}_X \to {\cal S}$.

\begin{defn} A functor $f$ is said to be {\it  excisive} if
for any collection of objects $Y_\alpha$, the natural map
$$
f(\coprod_\alpha Y_\alpha) \to \prod_{\alpha} f(Y_\alpha)
$$
is a weak equivalence.
\end{defn}

\begin{rem} This condition can be stated alternatively
as saying that $f$ preserves homotopy pushouts and that
up to weak equivalence, $f$ is determined up to equivalence by its
restriction to the full subcategory of 
finite complexes over $X$.

The last condition means  that
the natural map
$$
f(Y) \to \underset{Z \in C_Y }{\text{holim }} f(Z)
$$ 
is a weak equivalence, where the homotopy limit is
indexed over the category $C_Y$ consisting of spaces over $Y$
which are homeomorphic to a finite complex.           
\end{rem}

\begin{defn} A functor $f\: {\cal T}_X \to {\cal S}$
is {\it cohomological} if
\begin{itemize}
\item $f$ is a homotopy functor.
\item The value of $f$ at the initial object $\emptyset$ is
contractible.
\item $f$ is strongly excisive.
\end{itemize}
\end{defn}

\begin{thm}[Representability] \label{rep} For 
cohomological functors $f$,
 there is a fibered spectrum ${\cal E}$ and a natural equivalence
of functors
$$
f(Y)\,\,  \simeq \,\, H^\bullet(Y;{\cal E})
$$
\end{thm}

\begin{rems} (1). The fibered spectrum 
${\cal E}$ is unique up to equivalence. Heuristically, the 
value of $H^\bullet({-}; {\cal E})$ at
the one point maps $x\to X$ recovers
the fibers ${\cal E}_x$ of ${\cal E}$. The homotopy
colimit in the category of {\it un}based spaces of 
$({\cal E}_x)_j$ recovers the $j$-th total space ${\cal E}_j$
up to equivalence.

{\flushleft (2).} Our method of proof can be adapted
to show that the functor
$$
{\cal E} \mapsto H^\bullet({-};{\cal E})
$$
defines an equivalence the homotopy category of 
fibered spectra over $X$ and  the homotopy category
of cohomological functors. We will not need this statement.
\end{rems}

The main tool in the proof of \ref{rep} is
a natural transformation
$$
f(Y) \to f^\natural(Y)\, ,
$$
called the {\it coassembly map},
which is defined for any homotopy functor $f$. The target functor
$f^\natural$ is always strongly excisive. 

The idea will then be to show that the coassembly map is 
a weak equivalence when $f$ satisfies our assumptions, and
that $f^\natural$ is representable.

\subsection*{The coassembly map}
Let $f^\natural$ be the functor defined by
$$
f^\natural(Y) = \underset{\Delta^p \to Y}{\text{holim }} f(\Delta^p) \, ,
$$
where the homotopy limit is indexed over the category $\Delta_Y$  of singular
simplices in $Y$. This is the category whose objects are maps
$\Delta^p\to Y$  (for $p \ge 0$), where $\Delta^p$ is the standard $p$-simplex,
and morphisms are given by inclusions of faces. 

Given any map $\Delta^p \to Y$ we obtain a map $f(Y) \to f(\Delta^p)$.
This assignment is compatible with taking faces, so we get a natural map
$$
c: f(Y) \to f^\natural(Y)\, .
$$ 
This is the coassembly map.

We now verify the properties of $f^\natural$.
Note that $f^\natural$ is a homotopy functor since 
$f$ is and the homotopy limit
construction is homotopy invariant. Furthermore, $f^\natural$ 
is strongly excisive
because 
$$
\Delta_{\amalg_\alpha Y_\alpha} \,\, = \,\, \coprod_{\alpha} 
\Delta_{Y_{\alpha}}
$$
and the homotopy limit indexed over coproduct of categories 
is the product of the corresponding homotopy limits. 
Consequently, $f^\natural$ is strongly excisive.

\begin{proof}[Proof of Theorem \ref{rep}]
Assuming that $f$ is cohomological,
We first show that the coassembly map
$$
c: f(Y) \to f^\natural(Y)
$$
is a weak equivalence. It
 clearly is a weak equivalence when
$Y$ is the initial object. It is also a weak equivalence
when $Y$ is a point, since in this case the map $\Delta^p \to *$
is a weak equivalence and $f$ is a homotopy functor. Since $f$ is excisive,
$c$ is a weak equivalence when $Y$ is a finite set over $X$.

A Mayer-Vietoris
argument then shows that the coassembly map is a weak equivalence
whenever $Y$ is a finite complex over $X$.
Because $f$ is strongly excisive, this is enough to show that $c$ is a weak
equivalence in general, since $f$ is determined up to weak equivalence by
its restriction to the category of finite complexes over $X$.

To complete the proof of Theorem \ref{rep}, we will show that $f^\natural$
is representable. For $Y \in {\cal T}_X$, consider the functor
$$
f^Y_{j} \: \Delta_Y \to {\cal T}
$$
given as follows: for an object $\sigma\:\Delta^p \to Y$, set
$f^Y_j(\sigma) = 
f(\Delta^p)_j$, the $j$-th space of the spectrum $f(\Delta^p)$,
which we will consider here as an unbased space.
Define
$$
{\cal E(Y)}_j \,\, := \,\, \text{hocolim } f^Y_j\, .
$$
If we let $j$ vary, the ${\cal E}(Y)_j$ define a fibered spectrum
${\cal E(Y)}$.

By considering the constant map $f^Y_j(Y) \to *$ 
and taking homotopy colimits, we have a map
$$
{\cal E(Y)}_j \to B\Delta_Y
$$
where $B\Delta Y$ is
the classifying space of the category $\Delta_Y$, i.e,
 the geometric realization of its nerve
(recall that the homotopy colimit of 
the constant functor to a point is $B\Delta_Y$).
This map has the following properties.
\begin{itemize}
\item  It is a quasifibration, i.e., the map
from each fiber to its corresponding 
homotopy fiber is a weak equivalence.
\item The space of sections of the associated
fibration is weak equivalent to the homotopy limit of $f^Y_j$.
This is an observation of Dwyer
\cite[prop.\ 3.12]{Dwyer_centralizer}.
\item By definition, the collection 
$$
\{\text{holim } f^Y_j\}_{j \ge 0}
$$
is the spectrum  $f^\natural(Y)$.
\item 
$$
\xymatrix{
{\cal E(Y)}_j \ar[r]\ar[d]& {\cal E(X)}_j \ar[d] \\
 B\Delta_Y \ar[r] &  B\Delta_X
}
$$
is homotopy cartesian.
\end{itemize}

Set ${\cal E} := {\cal E}(X)$. Then ${\cal E}$ is a
fibered spectrum over $X$, and
it is a straightforward consequence of the above properties
that there is a natural weak equivalence
$f^\natural(Y) \simeq H^\bullet(Y;{\cal E})$. This completes the
proof of Theorem \ref{rep}.
\end{proof}

\end{document}